\documentclass[10pt]{amsart}%
\usepackage{amsmath}
\usepackage{amsfonts}
\usepackage{amssymb}
\usepackage{graphicx}
\usepackage{color}%
\providecommand{\U}[1]{\protect\rule{.1in}{.1in}}
\providecommand{\U}[1]{\protect\rule{.1in}{.1in}}

\theoremstyle{plain}
\numberwithin{equation}{section}
\newtheorem{theorem}{Theorem}[section]
\newtheorem{lemma}[theorem]{Lemma}
\newtheorem{proposition}[theorem]{Proposition}
\newtheorem{corollary}[theorem]{Corollary}

\theoremstyle{definition}

\theoremstyle{remark}
\newtheorem*{claim*}{Claim}
\newtheorem*{example*}{Example}
\newtheorem*{remark*}{Remark}
\newtheorem{remark}{Remark}

\begin{document}
\title[Parabolic Optimal Transport ]{Parabolic optimal transport equations on manifolds}
\author{Young-Heon Kim, Jeffrey Streets and Micah Warren}
\address{Department of Mathematics, University of British Columbia, Vancouver BC Canada\\
Department of Mathematics, Princeton University, Princeton NJ, USA}
\email{yhkim@math.ubc.ca, jstreets@math.princeton.edu,mww@princeton.edu}
\thanks{Y-H.K. is supported partly by NSERC discovery grant 371642-09. J.S. is
supported in part by NSF Grant 0703660.\ \ M.W. is supported in part by NSF
Grant DMS-0901644. \ Any opinions, findings and conclusions or recommendations
expressed in this material are those of authors and do not reflect the views
of either Natural Sciences and Engineering Research Council of Canada or the
US National Science Foundation. }
\thanks{\indent
\copyright 2010 by the authors. }
\date{\today}

\begin{abstract}
We study a parabolic equation for finding solutions to the optimal transport
problem on compact Riemannian manifolds with general cost functions. We show
that if the cost satisfies the strong MTW condition and the stay-away
singularity property, then the solution to the parabolic flow with any
appropriate initial condition exists for all time and it converges
exponentially to the solution to the optimal transportation problem. Such
results hold in particular, on the sphere for the distance squared cost of the
round metric and for the far-field reflector antenna cost, among others.

\end{abstract}
\maketitle

\section{Introduction}

\bigskip In this paper we study a fully nonlinear parabolic flow toward
optimal transport maps between smooth densities on compact manifolds. Let $M$,
$\bar{M}$ be two $n$-dimensional closed (i.e. compact and without boundary)
Riemannian manifolds {equipped with} the transportation cost { $c:M\times
\bar{M}\ni(x,\bar{x})\mapsto c(x,\bar{x})\in\mathbf{R}\cup\{+\infty\}$.} For
example, $M=\bar{M}$ and $c(x,\bar{x})=\mathrm{dist}(x,\bar{x})^{2}$ where
$\text{\textrm{dist}}$ denotes the Riemannian distance function of the given
metric. As we see the presence of cut-locus for $\text{\textrm{dist}}$, the
transportation cost may not be smooth, and we denote $\text{\textrm{sing}%
}(c)\subset M\times\bar{M}$ to be the set of points where $c$ fails to be
$C^{\infty}$. Let $\rho$, $\bar{\rho}$ be two smooth probability measures on
$M$, $\bar{M}$, respectively. (We use $\rho(x),\bar{\rho}(\bar{x})$ to express
the corresponding densities in local coordinates, i.e. $\rho=\rho(x)dx$,
$\bar{\rho}=\rho(\bar{x})d\bar{x}$.)

One seeks an optimal map $T:M\rightarrow\bar{M},$ which minimizes the cost
functional
\[
\int_{M}c(x,F(x))\rho dx
\]
among all maps $F$ which satisfy $F_{\#}\rho=\bar{\rho}$, i.e., $\int
_{M}f(F(x))\bar{\rho}(x)dx=\int_{\bar{M}}f(\bar{x})\bar{\rho}(\bar{x})d\bar
{x}$ for all $f\in C^{\infty}(\bar{M})$.

It is well known (c.f. \cite{B, GM, Mc, MTW}) that such an optimal map
uniquely exists under some appropriate conditions on $c$,$\rho,\bar{\rho}$,
and it always is associated with a potential function $u:M\rightarrow
\mathbf{R}$, such that the optimal map $T$ (if smooth) satisfies
\begin{equation}
-D_{x}c(x,T(x))=D_{x}u(x).\label{eq:potential}%
\end{equation}
Here, $D_{x}$ denotes the derivative in the first variable $x\in M$. Note that
the condition $T_{\#}\rho=\bar{\rho}$ forces the potential $u$ (if smooth) to
satisfy a Monge-Amp\`{e}re type equation%
\begin{equation}
\det\left(  u_{ij}(x)+c_{ij}(x,T(x)\right)  =\frac{\rho(x)}{\bar{\rho}%
(T(x))}\left\vert \det c_{i\bar{j}}(x,T(x))\right\vert \label{eq:OT}%
\end{equation}
where the subindices $i,j,\cdots$ denote the derivative (in local coordinates)
in the first variable $x$ while the barred subindices $\bar{\imath},\bar
{j},\cdots$ denote the one for the second variable $\bar{x}$.

The seminal work of Ma, Trudinger and Wang \cite{MTW, TW} (which follows that
of Delano\"{e} \cite{D}, Caffarelli \cite{C, C2} and Urbas \cite{U}) studies
regularity of optimal maps for general cost functions $c$ by considering
regularity of solutions to the fully nonlinear (degenerate) elliptic equation
\eqref{eq:OT}. Importantly, they{ have} identified a structure condition on
$c$, now widely called the $\mathsf{MTW}$ condition (see \ref{SS:MTW} for more
details), which later is shown by Loeper \cite{Lo} to be a necessary condition
for the regularity of solutions to \eqref{eq:OT}. This then is followed by
many works including those of (in alphabetical order) Delano\"{e}, Figalli,
Ge, Kim, Liu, Loeper, McCann, Rifford, Trudinger, Villani, Wang and Warren
\cite{Lo, Lo2, TW2, KM1, Km, KM2, DG, LV, Villani, FR, LTW, FKM, kimmccwar,
FKM-econ, FKM-product, FRV-reg}, among others, which study both regularity of
optimal maps and geometric issues related to the $\mathsf{MTW}$ condition.

With this in mind, in the present paper we study the following fully nonlinear
parabolic equation for $u:M\times\lbrack0,\infty)\rightarrow\mathbf{R}$,
\begin{equation}
\frac{\partial u}{\partial t}=\ln\det(u_{ij}+c_{ij}(x,T(x)))-\ln\rho
(x)+\ln\bar{\rho}(T(x))-\ln\det|c_{i\bar{j}}(x,T(x))|. \label{eq:main}%
\end{equation}
Here, to define $T$ uniquely by $u$ through \eqref{eq:potential} and to make
the logarithm make sense, we assume (as in \cite{MTW}) throughout the present
paper, the following conditions for the cost function:
\begin{align}
&  \text{the map $\bar{x}\mapsto D_{x}c(x,\bar{x})\in T_{x}^{\ast}M$ is
injective for $(x,\bar{x})\in M\times\bar{M}\setminus\text{\textrm{sing}}(c)$%
};\label{eq:a2}\\
&  \text{$\det c_{i\bar{j}}(x,\bar{x})\neq0$ on $M\times\bar{M}\setminus
\text{\textrm{sing}}(c)$. } \label{eq:a29}%
\end{align}
Notice that the right hand side in \eqref{eq:main} is a coordinate invariant quantity.

The main result of the present paper is the following

\begin{theorem}
[Parabolic flow toward optimal transport on manifolds]\label{main} Let $M$,
$\bar{M}$ be two $n$-dimensional compact Riemannian manifolds without
boundary, equipped with a cost function { $c:M\times\bar{M}\rightarrow
\mathbf{R}\cup\{+\infty\}$} satisfying \eqref{eq:a2} and \eqref{eq:a29}. {
Assume that $c$ is locally semi-concave on the domain where its value is
finite.} Let $\rho,\bar{\rho}$ be smooth, positive, probability densities on
$M$, $\bar{M}$, respectively. Assume further $c$ satisfies (i)
$\mathsf{MTW(\delta)}$ condition for some $\delta>0$, and (ii)
stay-away-from-singularity property (see Sections \ref{SS:MTW} and
\ref{SS:stay away} for definitions). {Let $u_{0}:M\rightarrow\mathbf{R}$ be a
$C^{2}$ locally strictly $c$-convex function: see \eqref{eq:locally c-convex}
for definition. Then,} there exists a unique smooth solution $u:M\times
\lbrack0,\infty)\rightarrow\mathbf{R}$, of \eqref{eq:main} with $u(\cdot
,0)\equiv u_{0}$. \ Moreover, the derivatives of $u$ (i.e. $\Vert
u\Vert_{C^{m}}$ for each $m\in%
\mathbb{N}
$) are bounded uniformly in time, and as $t\rightarrow\infty$, $u(\cdot,t)$
converges exponentially to a solution of the equation \eqref{eq:OT}, \ and
this solution to \eqref{eq:OT} defines the unique solution to the optimal
transportation problem. \ \ 
\end{theorem}

\begin{remark}
In fact, as a corollary to Theorem~\ref{main}, one can show that $C^{2}$
locally strictly $c$-convex functions are indeed globally $c$-convex in the
situation of Theorem~\ref{main}. See
Corollary~\ref{C:local to global c-convex} (see Section~\ref{SS:stay away} for definitions).
\end{remark}

The main consequences of Theorem~\ref{main} are in the following

\begin{corollary}
\label{C:main cor} The same existence and exponential convergence result for
the solution to \eqref{eq:main} as in Theorem~\ref{main} holds for the
following cases of $M=\bar{M}$, where the condition $\mathsf{MTW}(\delta)$,
$\delta>0$, and the stay-away-from-singularity property as well as regularity
of solutions to the elliptic equation \eqref{eq:OT} are shown in the papers
cited correspondingly :

\begin{itemize}
\item[(1)] $c=\text{\textrm{dist}}^{2}/2$ on the round sphere (see \cite{DL,
Lo2} for elliptic case);

\item[(2)] $c=\text{\textrm{dist}}^{2}/2$ on small $C^{4}$ perturbations of
the round sphere, but in a restricted sense that $\lambda_{1}, \lambda_{2}$ in
\eqref{eq:stay away} are restricted by the size of perturbation (see \cite{DG,
FRV} for elliptic case);

\item[(3)] $c=\text{\textrm{dist}}^{2}/2$ on the covering quotient
$S^{n}/\Gamma$, $\Gamma$ discrete group, e.g. $\mathbf{RP}^{n}$, (see
\cite{DG, KM2} for elliptic case), and their $C^{4}$ perturbations (see
\cite{LV} for elliptic case).

\item[(4)] $c=\text{\textrm{dist}}^{2}/2$ on the complex project space
$\mathbf{CP}^{n}$ and the quaternionic $n$-space $\mathbf{HP}^{n}$ with the
metric induced from the round sphere by the Hopf fibrations $S^{2n+1}
\to\mathbf{CP}^{n}$ and $S^{4n+3} \to\mathbf{HP}^{n}$ (see \cite{KM2} for
elliptic case).

\item[(5)] the far-field reflector antenna cost $c(x,\bar{x})=-\log|x-\bar
{x}|$ on the imbedded round sphere $S^{n}\subset\mathbf{R}^{n+1}$
(Schn\"{u}rer \cite{Sch} gave a proof in parabolic case, see \cite{Lo2} for
elliptic case.)
\end{itemize}
\end{corollary}

\begin{remark}
In all of the above cases one can take $u_{0}\equiv0$ as an initial condition.
\end{remark}

\begin{remark}
Recently, in \cite{FKM-product} regularity of solutions to the elliptic
equation \eqref{eq:OT} is shown for $c=\text{\textrm{dist}}^{2}/2$ on the
multiple products $S_{r_{1}}^{n_{1}}\times\cdots\times S_{r_{k}}^{n_{k}}$ of
round spheres of arbitrary dimension $n_{i}$ and size $r_{i}$, by showing the
stay-away-from-singularity property and using the results in \cite{FKM, LTW}.
However, we do not have the corresponding parabolic version yet.
\end{remark}

{ The proof of Theorem~\ref{main} is given in the rest of the paper;
especially, see Sections~\ref{S:long time} and \ref{S:exponential}. We use the
tensor maximum principle method for obtaining the second derivative estimates,
which} is essentially the calculation of Ma, Trudinger and Wang \cite{MTW},
given the stay away assumption. \ We use the parabolic Krylov-Safonov theory
to obtain { $C^{2,\alpha}$ estimates from $C^{1,1}$ estimates}, although we
note that a long and straightforward parabolic adaptation of the Calabi
computation in \cite{CNS} would furnish the $C^{1,1\text{ }}$to $C^{2,1}$
estimates, if one were interested in a \textquotedblleft bare-hands" proof.
\ Higher estimates follow by the Schauder theory. { In order to obtain the
exponential convergence, instead of relying on the quite general
Krylov-Safonov Harnack estimates, we directly prove Li-Yau type Harnack
estimates. Also, a simple topological argument is used to show that the
resulting limit function at $t=\infty$ is indeed the solution to the optimal
transport problem.}

The proof of Li-Yau {type} estimates illustrates the natural relevance of the
Riemannian submanifold geometry that is first seen in \cite{KM1}, further
refined in \cite{kimmccwar} and then again in \cite{Warren}. \ In this paper,
we see that the quantity $\theta$, defined{ to be the right-hand side of
(\ref{eq:main}),} which measures how far the map at a given time is from being
volume preserving, satisfies a heat-like equation, which makes it vulnerable
to the Li-Yau approach. For nonlinear equations, \ the idea of using the
linearized operator to define{ a Laplacian with respect to a} metric goes back
to Calabi \cite{Cb}, and also arises naturally when studying minimal
submanifold equations. \ In our case, there is an interesting conformal
factor, which forces us to perform a slight workaround in two dimensions.\ See
Section \ref{S:exponential convergence}.

Having practical applications in mind, {the exponential convergence in
Theorem~\ref{main} can be of independent interest since} one can view our
parabolic flow as an algorithm to construct the solution to optimal transport
problem. Indeed, Schnurer \cite{Sch} applied a parabolic flow to construct
reflector antennas for given light sources: this practical problem in
geometric optics has been known to be an optimal transport problem as found by
X.-J. Wang \cite{Wang96}. Recently, Kitagawa \cite{Kt} has a result similar to
ours, which deals with oblique boundary value problem on domains in
$\mathbf{R}^{n}$, for cost functions satisfying the $\mathsf{MTW}%
(0)$\ condition (see Section~\ref{SS:MTW}). \ 

The paper is organized as follows. First, after a few preliminaries in
Section~\ref{S:pre}, we obtain in Section~\ref{S:estimates} the
stay-away-from-singularity property along the flow which is then used together
with the maximum principle argument of Ma, Trudinger, and Wang \cite{MTW}
applied to the parabolic setting, to get uniform second derivative estimates
along the flow. In Section~\ref{S:long time} {we use these estimates to obtain
the long time existence and uniform derivative bounds of the solution.
Following some remarks about the linearized operator} in Section
\ref{S:exponential convergence}, we prove the Harnack inequality in Section
\ref{S:harnack}. \ Finally, in Section~\ref{S:exponential}, the exponential
convergence to the stationary solution at infinity is obtained, {and this
stationary solution is shown to be the solution to the optimal transport
problem, thus finishing the proof of Theorem~\ref{main}. As a corollary to
Theorem~\ref{main}, it is shown in Corollary~\ref{C:local to global c-convex}
that $C^{2}$ locally $c$-convex functions are globally $c$-convex.}

\textbf{Acknowledgement:} Y.-H.K. is pleased to thank Albert Chau for helpful
discussions. \ M.W. is pleased to thank Jun Kitagawa and Oliver Schn\"{u}rer
for helpful conversations. \ 

\section{Preliminaries}

\label{S:pre}

\subsection{ Ma-Trudinger-Wang curvature}

\label{SS:MTW} We first explain the $\mathsf{MTW}$ curvature that is first
introduced by Ma, Trudinger and Wang \cite{MTW} as a quanitity which can be
used to gaurantee interior regularity for solutions of the elliptic {equation
\eqref{eq:OT}}. This notion is further investigated by Loeper \cite{Lo} and
then by Kim and McCann \cite{KM1} (see also \cite{Km, KM2, LV, Villani, FR,
FRV-reg}). As shown by \cite{KM1}, the $\mathsf{MTW}$ tensor associated to
$c:M\times\bar{M}\rightarrow\mathbf{R}$, can be understood as the curvature of
a pseudo-metric
\begin{equation}
h=-\frac{1}{2}c_{i\bar{j}}\left(  dx^{i}\otimes d\bar{x}^{\bar{j}}+d\bar
{x}^{\bar{j}}\otimes dx^{i}\right)  \label{kmmetric}%
\end{equation}
defined on $M\times\bar{M}\setminus\text{\textrm{sing}}(c)$. For $(V,\bar
{V},W,\bar{W})\in T_{(x,\bar{x})}(M\times\bar{M})\times$ $T_{(x,\bar{x}%
)}(M\times\bar{M})$ (i.e. $V,W\in T_{x}M$, $\bar{V},\bar{W}\in T_{\bar{x}}%
\bar{M}$), it is computed in local coordinates \cite{MTW} as
\[
\mathsf{MTW}(V,\bar{V},W,\bar{W})=(-c_{ij\bar{k}\bar{l}}+c_{ij\bar{a}}%
c^{\bar{a}b}c_{\bar{k}\bar{l}b})V^{i}W^{j}\bar{V}^{\bar{k}}\bar{W}^{\bar{l}}.
\]

\ At this point we mention the $\left(  2,2\right)  $ form of the tensor,
which appeared first in {\cite{Tr}}. \ While the coordinate invariance
implicit in {\cite{Tr}} is easily checked directly by an elementary
calculation (without refering to curvature), we mention how it comes from
(\ref{kmmetric}). \ Due to the structure of the metric (\ref{kmmetric}) the
sharp operator, { which identifies covectors with vectors,} on $M\times\bar
{M}$ is a actually a map from $T^{\ast}M$ to $T\bar{M}$. \ Thus we have the
following $\left(  2,2\right)  $ tensor:
\begin{align}
\mathsf{MTW}  &  :TM\times T^{\ast}M\times TM\times T^{\ast}M\rightarrow
\mathbb{R}\nonumber\\
\mathsf{MTW}(V,\eta,W,\zeta)  &  ={\left(  -c_{ij\bar{r}\bar{p}}+c_{ij\bar{s}%
}c^{\bar{s}m}c_{m\bar{p}\bar{r}}\right)  V^{i}W^{j}c^{\bar{p}k}c^{\bar{r}%
l}\eta_{k}\zeta_{l}.\label{MTWm}}%
\end{align}
This formulation of the tensor arises naturally in the MTW\ calcuation and has
the advantage of having quantities all in terms of vectors and covectors on
$M.$ \ We say that $c$ satisfies$\ \mathsf{MTW}(\delta)$\ condition with
respect to a metric $g$ on $M$ if, for $\delta>0$
\begin{equation}
\mathsf{MTW}(V,\eta,V,\eta)\geq\delta\left\Vert V\right\Vert _{g}
^{2}\left\Vert \eta\right\Vert _{g}^{2} \label{A3strong}%
\end{equation}
for all vector covector pairs with $\eta(V)=0.$ The metric $g$ on the source
needs nothing to do with the cost function, but by fixing the metric,
(\ref{A3strong}) becomes invariant.

\subsection{{Stay away from singularity property for a cost function and
$c$-convexity of potential functions}}

\label{SS:stay away} Fix Riemannian metrics $g$, $\bar{g}$ on $M$, $\bar{M}$,
respectively. We say that a cost { $c:M\times\bar{M}\rightarrow\mathbf{R}%
\cup\{+\infty\}$} has the \emph{stay-away-from-singularity property} if for
each $0<\lambda_{1},\lambda_{2}$, there exists $\epsilon>0$ depending only on
$\lambda_{1}$, $\lambda_{2}$ and $c$ such that
\begin{equation}
\lambda_{1}\leq|\det DT|\leq\lambda_{2}\Longrightarrow\text{\textrm{dist}%
}(\mathrm{graph}(T),\text{\textrm{sing}}(c))\geq\epsilon\label{eq:stay away}%
\end{equation}
for any differentiable map $T:M\rightarrow\bar{M}$ given by
\eqref{eq:potential} with {$C^{2}$} locally strictly $c$-convex potential
function $u:M\rightarrow\mathbf{R}$. Here, the {$C^{2}$} potential function
$u$ is called \emph{locally strictly $c$-convex} if
\begin{align}
\label{eq:locally c-convex}\text{ $u_{ij}(x)+c_{ij}(x,T(x))$ is positive
definite for each $x\in M$.}%
\end{align}
{This local strict $c$-convexity on $u$ is equivalent to that the map $T$ is a
local diffeomorphism.} In the above, $\det DT$ is computed with respect to the
metrics $g$ and $\bar{g}$, and the distance function $\text{\textrm{dist}}$ is
with respect to the product metric $g\oplus\bar{g}$ on $M\times\bar{M}$.

{To be a solution to the optimal transportation problem for $\rho$ and
$\bar{\rho}$, a solution $u:M\rightarrow\mathbf{R}$ of \eqref{eq:OT} has to be
\emph{a global $c$-convex function}. Namely, $u$ is given as a pair
$(u,\bar{u})$ as
\[
u(x)=\sup_{\bar{x}\in\bar{M}}-c(x,\bar{x})-\bar{u}(\bar{x}),\qquad\bar{u}%
(\bar{x})=\sup_{x\in M}-c(x,\bar{x})-u(x)
\]
for all $(x,\bar{x})\in M\times\bar{M}.$ If $M$ is a closed manifold and the
cost function $c$ is locally semi-concave on the set where its value is
finite, then for any $C^{2}$ locally strictly $c$-convex function $u$ on $M$,
the \emph{global $c$-convexity} is implied if the corresponding map $T$ via
\eqref{eq:potential} is a global diffeomorphism. To see this, suppose $u$ is
not globally $c$-convex. Then, there exists $x_{0}\in M$ and $x_{1}\neq x_{0}$
such that
\[
u(x_{1})+c(x_{1},T(x_{0}))<u(x_{o})+c(x_{0},T(x_{0}))
\]
Thus on the closed manifold $M$, there is an absolute minimum point $z_{0}\neq
x_{0}$, of the function $u(\cdot)+c(\cdot,T(x_{0}))$. Near $z_{0}$ the cost
$c(\cdot,T(x_{0})$ should be finite by the minimum property, thus locally
semi-concave by the assumption on $c$. Since $u$ is $C^{2}$, the function
$u(\cdot)+c(\cdot,\bar{y})$ is semi-concave too, hence superdifferentiable.
Thus, this function cannot achieve a minimum at points of
nondifferentiability, so we conclude that $u(\cdot)+c(\cdot,T(x_{0}))$ is
differentiable at $z_{0}$. Therefore,
\begin{align*}
Du(z_{0})+c(z_{0},\bar{y}) &  =0\\
Du(x_{0})+c(x_{0},\bar{y}) &  =0.
\end{align*}
In particular the local diffeomorphism $T$ is not one-to-one, showing the
claimed equivalence.}

The stay-away-from-singularity property is shown for the round sphere
$M=\bar{M}=S^{n}$ with $c=\text{\textrm{dist}}^{2}/2$ by Delano\"{e} and
Loeper \cite{DL}, and later also for the reflector antenna cost $c=-\log
|x-\bar{x}|$ on $S^{n}\subset\mathbf{R}^{n+1}$ \cite{Lo2} , and for
$c=\text{\textrm{dist}}^{2}/2$ on the perturbations of the round sphere and
its discrete quotient \cite{DG, LV, KM2}, on the Hopf fibration quotients of
the sphere such as $\mathbf{CP}^{n}$ and $\mathbf{HP}^{n}$ \cite{KM2}.
{Recently, such result is also shown for $c=\text{\textrm{dist}}^{2}/2$ on the
products of round spheres $M=\bar{M}=S_{r_{1}}^{n_{1}}\times\cdots\times
S_{r_{k}}^{n_{k}}$ of arbitrary dimension and size \cite{FKM-product}. }

\begin{remark}
Strictly speaking, the stay-away results in the papers \cite{DL, Lo2, DG, LV,
KM2, FKM-product} are shown with respect to globally $c$-convex functions,
which in general differ from locally $c$-convex ones. However, in these cases,
one can actually prove such stay-away property with respect to locally strict
$c$-convex functions. Alternatively, it can be shown that in those cases,
locally $c$-convex functions are globally $c$-convex, using the results of
\cite{TW2, Villani, TW3, FRV-reg}.
\end{remark}

\section{Estimates}

\label{S:estimates} In the following, $T$ will always denote the map given by
the relation \eqref{eq:potential}, which is equivalent in any coordinate chart
$x=(x^{1},\cdots,x^{n})$, to
\begin{equation}
u_{i}(x)+c_{i}(x,T(x))=0. \label{contact}%
\end{equation}
Further differentiation gives
\[
u_{ij}+c_{ij}(x,T)+c_{i\bar{s}}(x,T)T_{j}^{\bar{s}}=0.
\]
Here and henceforth we use $\bar{r},\bar{s},\bar{p},\bar{v}$, etc to denote
differentiation in the second variable of $c(x,\bar{x})$, with the map $T$
represented in coordinates as $T^{\bar{s}}$, $\bar{s}=1,\cdots,n$. Taking the
determinant of the above equation gives the elliptic equation (\ref{eq:OT}). A
useful observation is that the matrix $w_{ij}(x):=u_{ij}(x)+c_{ij}(x,T(x))$
(when positive definite) gives a Riemannian metric on $M$ by the following
identity:
\[
(Id\times T)^{\ast}h(\frac{\partial}{\partial x^{i}},\frac{\partial}{\partial
x^{j}})=w_{ij}%
\]
where $(Id\times T)^{\ast}h$ is the pull-back of $h$ by $Id\times T:M\ni
x\mapsto(x,T(x))\in M\times M$.

\ \ For future reference we note the following results of differentiating
\qquad%
\begin{equation}
T_{j}^{\bar{s}}=-c^{i\bar{s}}w_{ij}, \label{wT}%
\end{equation}%
\begin{equation}
T_{jk}^{\bar{s}}=-c^{i\bar{s}}\frac{\partial}{\partial x^{k}}w_{ij}%
+c^{i\bar{p}}c^{l\bar{s}}(c_{l\bar{p}k}+c_{l\bar{p}\bar{r}}T_{k}^{\bar{r}%
})w_{ij}. \label{diff2}%
\end{equation}
Here, to be clear, we write $\frac{\partial}{\partial x^{k}}w_{ij}$ to{
express coordinate derivatives} of this tensor. \ Subscripts for other
functions mean the corresponding coordinate derivatives. Superscripts such as
$c^{i\bar{s}}$ denotes the $(i,\bar{s})$-entry of the inverse matrix of
$\left(  \frac{\partial^{2}}{\partial x^{a}\bar{x}^{\bar{b}}}c\right)  $, etc.
\ Defining%
\begin{equation}
\theta(x,u)=\ln\det w_{ij}-\ln\rho(x)+\ln\bar{\rho}(T(x))-\ln|\det c_{i\bar
{s}}(x,T(x))|, \label{theta}%
\end{equation}
the flow (\ref{eq:main}) is rewritten as
\begin{equation}
u_{t}=\theta. \label{1}%
\end{equation}
Here and henceforth, the subscript $t$ denotes the time derivative.

\subsection{Stay away from singularity}

\label{SS:stay away estimates}

We make use of the linearized operator $L$ of $\theta$ at a function $u$. {
This operator $L$ is covariant, and as} seen from [TW] the operator $L$ has
the following local expression.
\begin{align}
\label{eq:L}Lv=w^{ij}v_{ij}-\left(  w^{ij}c_{ij\bar{s}}c^{\bar{s}k}-(\bar
{\rho})^{-1}\bar{\rho}_{\bar{r}}c^{k\bar{r}}-c^{i\bar{s}}c_{\bar{s}i\bar{p}%
}c^{\bar{p}k}\right)  v_{k}.
\end{align}
where all coefficients are computed at $(x,T(x))$. It is important to notice
that there is no zeroth-order term in $L$ so that we can apply parabolic
maximum principle.

\begin{lemma}
\label{thetaevolution} If $u$ is a solution of (\ref{1}), then $\theta$
satisfies
\begin{align}
\label{eq:theta flow}\theta_{t}=L\theta.
\end{align}

\end{lemma}

\begin{proof}
This follows from the definition of linearized operator $\frac{\partial
}{\partial t} \left(  \theta(x, u)\right)  = L \frac{\partial}{\partial t} u$
and using the equation \eqref{1}.
\end{proof}

\begin{proposition}
[Stay-away-from-singularity property along the flow]\label{P:stay away} Let
$M$, $\bar{M}$ be two $n$-dimensional compact Riemannian manifolds without
boundary, equipped with cost function $c:M\times\bar{M}\rightarrow\mathbf{R}$
satisfying \eqref{eq:a2} and \eqref{eq:a29}. Let $\rho,\bar{\rho}$ be smooth,
positive, probability densities on $M$, $\bar{M}$, respectively. Suppose the
cost $c$ satisfies the stay-away-from-singularity property (see
Section~\ref{SS:stay away}). There exists a constant
\[
\varepsilon=\varepsilon\left(  \theta\big|_{t=0},\Vert\ln\rho\Vert_{\infty
},\Vert\ln\bar{\rho}\Vert_{\infty}\right)  >0,
\]
such that in the time interval of existence of the smooth solution $u$ to the
parabolic equation \eqref{eq:main}, $\text{\textrm{dist}}(\text{\textrm{graph}%
}(T),\text{\textrm{sing}}(c))\geq\varepsilon$ for $T$ given by
\eqref{eq:potential}. Here, $\varepsilon$ is independent of the time $t$.
\end{proposition}

\begin{proof}
It follows from Lemma 3.1 and the parabolic maximum principle applied to
\eqref{eq:theta flow} that $\theta$ is bounded along the flow. Now as
\[
\theta(x,t)=\ln\det DT-\ln\rho(x) - \ln\bar\rho(T (x)),
\]
a bound on $\|\ln\det DT\|_{\infty}$ \ follows from the bound on $\theta.$
This together with the stay-away-property \eqref{eq:stay away} completes the proof.
\end{proof}

\begin{corollary}
[Uniform bounds on derivatives of $c$ along the flow]\label{C:uniform bounds}
Suppose the same assumption as in Proposition~\ref{P:stay away} holds. Then,
in the time interval of existence of the smooth solution $u$ to the flow of
\eqref{eq:main}, each derivative of $c$ in any order such as $c_{ij}%
,c_{i\bar{s}},c_{ijk},c_{ij\bar{k}},$ etc. at $(x,T(x))$ computed in any fixed
coordinates, is uniformly bounded, especially independent of $t$. Moreover,
$c_{i\bar{s}}$ is uniformly away from $0$.
\end{corollary}

{ }

\subsection{Second derivative estimates}

\label{SS:second derivative estimates}

On{ the both source} and target manifolds, we work in normal coordinates
throughout this section. Noting that the transition functions between charts
change in a predictible way, an $\varepsilon/4$ covering argument can be used
to show there is a uniform bound on (in particular third and fourth)
derivatives of the cost function which doesn't depend on the charts,
regardless of the point at the origin. \ Thus throughout this subsection, we
assume the result of Corollary~\ref{C:uniform bounds}, which is an essential
ingredient in the following estimates. Following the calculations given in
\cite{MTW} we show

\begin{lemma}
\label{secondderivativelemma} Assume the uniform bounds in
Corollary~\ref{C:uniform bounds}. Let $u$ be a solution of (\ref{1}), and let
$v$ be a coordinate direction in a local chart. \ Then
\begin{align}
\frac{\partial}{\partial t}w_{vv}-\left(  {Lw}\right)  {_{vv}}  &  \leq
-w^{ij}\left(  -c_{ij\bar{p}\bar{r}}+c_{ij\bar{s}}c^{\bar{s}m}c_{m\bar{r}%
\bar{p}}\right)  c^{\bar{p}k}c^{\bar{r}l}w_{lv}w_{kv}\label{mtwlemma}\\
&  +C(\rho,\bar{\rho},D^{4}c,g)(1+\sum w_{ii}^{2}+\sum w^{ii}\sum
w_{jj}).\nonumber
\end{align}
Here $L$ is operating on the $(2,0)$ tensor using metric covariant
differentiation. \ \ 
\end{lemma}

\begin{proof}
To simplify the calculation choose normal coordinates around $x$ such that
$v=\frac{\partial}{\partial x_{1}}$. \ Now{ from \eqref{eq:L}, }
\[
\left(  Lw\right)  _{11}=w^{ij}(w_{11,ij}-w_{ij,11})+w^{ij}w_{ij,11}-\left(
w^{ij}c_{ij\bar{s}}c^{\bar{s}k}-(\bar{\rho})^{-1}\bar{\rho}_{\bar{r}}%
c^{k\bar{r}}-c^{i\bar{s}}c_{\bar{s}i\bar{p}}c^{\bar{p}k}\right)  \partial
_{k}w_{11},
\]
with
\begin{align*}
w_{ij,11} &  =u_{ij11}+c_{ij11}+2c_{ij\bar{s}1}T_{1}^{\bar{s}}+c_{ij\bar{s}%
}T_{11}^{\bar{s}}+c_{ij\bar{p}\bar{r}}T_{1}^{\bar{p}}T_{1}^{\bar{r}}%
+w\ast\partial^{2}g\\
w_{11,ij} &  =u_{11ij}+c_{11ij}+c_{11\bar{s}i}T_{j}^{\bar{s}}+c_{11\bar{s}%
j}T_{i}^{\bar{s}}+c_{11\bar{s}}T_{ij}^{\bar{s}}+c_{11\bar{p}\bar{r}}%
T_{i}^{\bar{p}}T_{j}^{r}+w\ast\partial^{2}g
\end{align*}
where $w\ast\partial^{2}g$ denotes a linear combination of second derivatives
of the metric with components of $w.$ \ Thus
\begin{align*}
\left(  Lw\right)  _{11}\geq-K &  +w^{ij}(c_{11\bar{s}}T_{ij}^{\bar{s}%
}+c_{11\bar{p}\bar{r}}T_{i}^{\bar{p}}T_{j}^{r}-c_{ij\bar{s}}T_{11}^{\bar{s}%
}-c_{ij\bar{p}\bar{r}}T_{1}^{\bar{p}}T_{1}^{\bar{r}})\\
&  +w^{ij}\partial_{11}w_{ij}-\left(  w^{ij}c_{ij\bar{s}}c^{sk}-\left(
\ln\bar{\rho}\right)  _{\bar{r}}c^{k\bar{r}}-c^{i\bar{s}}c_{\bar{s}i\bar{p}%
}c^{\bar{p}k}\right)  \partial_{k}w_{11}.
\end{align*}
Here as in \cite{MTW} we let
\[
K=C\sum w^{ii}\sum w_{jj}+C\sum w^{ii}+C\sum w_{ii}^{2}.
\]
Now {use (\ref{diff2}) to see}%
\begin{equation}
T_{11}^{\bar{s}}=-c^{k\bar{s}}w_{k1,1}+c^{k\bar{p}}c^{l\bar{s}}(c_{l\bar{p}%
1}+c_{l\bar{p}\bar{r}}T_{1}^{\bar{r}}))w_{k1}\label{diff22}%
\end{equation}
and
\begin{equation}
w^{ij}T_{ij}^{\bar{s}}=-c^{k\bar{s}}w^{ij}(\partial_{j}w_{ki}-\partial
_{k}w_{ij})-c^{k\bar{s}}w^{ij}\partial_{k}w_{ij}+c^{j\bar{p}}c^{l\bar{s}%
}(c_{l\bar{p}j}+c_{l\bar{p}\bar{r}}T_{j}^{\bar{r}}).
\end{equation}
{Note that}
\[
\partial_{j}w_{ki}-\partial_{k}w_{ij}=c_{ki\bar{s}}T_{j}^{\bar{s}}%
-c_{ij\bar{s}}T_{k}^{\bar{s}}%
\]
{and use this to} cancel one $\partial_{k}w_{11}$ term to get%

\begin{align*}
\left(  Lw\right)  _{11}\geq-K-c_{11\bar{s}}c^{ks}w^{ij}\partial_{k}w_{ij}+ &
w^{ij}(-c_{ij\bar{s}}c_{l\bar{p}\bar{r}}T_{1}^{\bar{r}}w_{k1}-c_{ij\bar{p}%
\bar{r}}T_{1}^{\bar{p}}T_{1}^{\bar{r}})\\
&  +w^{ij}\partial_{11}w_{ij}+\left(  \left(  \ln\bar{\rho}\right)  _{\bar{r}%
}c^{k\bar{r}}+c^{i\bar{s}}c_{\bar{s}i\bar{p}}c^{\bar{p}k}\right)  \partial
_{k}w_{11}.
\end{align*}
Now using
\begin{align*}
\theta_{11} &  =w^{ij}\partial_{11}{w_{ij}+\partial}_{1}{w^{ij}{\partial}%
_{1}w}_{ij}+\left(  \ln\bar{\rho}\right)  _{\bar{s}}T_{11}^{s}+c^{i\bar{s}%
}c_{i\bar{s}\bar{p}}T_{11}^{\bar{p}}+c^{i\bar{s}}c_{i\bar{s}\bar{p}\bar{r}%
}T_{1}^{\bar{p}}T_{1}^{\bar{r}}-K\\
&  \geq w^{ij}\partial_{11}{w_{ij}+}\left(  \left(  \ln\bar{\rho}\right)
_{\bar{s}}{+c}^{is}c_{i\bar{s}\bar{p}}\right)  c^{k\bar{s}}w_{k1,1}-K
\end{align*}
we have%
\begin{align*}
\left(  Lw\right)  _{11}\geq-K &  -c_{11\bar{s}}c^{k\bar{s}}\theta_{k}%
+\theta_{11}\\
&  -w^{ij}\left(  -c_{ij\bar{s}}c^{l\bar{s}}c_{l\bar{p}\bar{r}}+c_{ij\bar
{p}\bar{r}}\right)  T_{1}^{\bar{r}}T_{1}^{p}%
\end{align*}
Now, finally we compute
\[
\frac{\partial}{\partial t}w_{11}=\theta_{11}+c_{11\bar{s}}T_{t}^{\bar{s}}.
\]
Differentiating (\ref{contact}) with respect to $t$ we get
\[
u_{kt}+c_{k\bar{s}}T_{t}^{\bar{s}}=0,
\]
which, subtracting, we arrive at
\[
\left(  Lw\right)  _{11}-w_{11t}\geq-w^{ij}\left(  -c_{ij\bar{s}}c^{l\bar{s}%
}c_{l\bar{p}\bar{r}}+c_{ij\bar{p}\bar{r}}\right)  T_{1}^{\bar{r}}T_{1}^{p}-K
\]
which is the conclusion of the Lemma using (\ref{wT}) . \ 
\end{proof}

\begin{theorem}
\label{Hessian} For a closed Riemannian manifold $M$ and an interval $[0,l)$,
suppose that $u:M\times\lbrack0,l)\rightarrow\mathbf{R}$ is a smooth solution
to the parabolic equation \eqref{1}. Assume that $c$ satisfies (i) the
$\mathsf{MTW}(\delta)$ condition, and (ii) the stay-away-from-singularity
property (so that the uniform bounds of Corollary~\ref{C:uniform bounds} hold
on $M\times\lbrack0,l)$). Let $v\in TM$, $|v|=1$, be such that
\[
w_{vv}=\max_{z\in TM,|z|=1}w_{zz}.
\]
There is a constant $C=C(\delta,c,n,g)$, especially independent of $t$, such
that if $w_{vv}(x,t)\geq C$ then
\[
w_{vvt}(x,t)\leq0.
\]
{ In particular, $w_{zz}(x,t)\leq C$ for any $(x,t)\in M\times\lbrack0,l)$ and
$|z|=1$. }
\end{theorem}

\begin{proof}
We use Hamilton's \cite{Ha} parabolic maximum principle argument for tensors.
\ Analyze the first term in the right hand side of (\ref{mtwlemma})
\begin{equation}
-w^{ij}\mathsf{MTW}_{ij}^{kl}w_{lv}w_{kv} \label{mtwsum}%
\end{equation}
where
\[
\mathsf{MTW}_{ij}^{kl}=\left(  -c_{ij\bar{p}\bar{r}}+c_{ij\bar{s}}c^{\bar{s}%
m}c_{m\bar{r}\bar{p}}\right)  c^{\bar{p}k}c^{\bar{r}l}.
\]
Diagonalize $w$ with $v=\frac{\partial}{\partial x_{1}},$ \ and (\ref{mtwsum}%
)\ becomes\
\[
-w^{11}\mathsf{MTW}_{11}^{11}w_{11}^{2}-\sum_{i>1}w^{ii}\mathsf{MTW}_{ii}%
^{11}w_{11}^{2}.
\]
By the MTW $(\delta)$\ condition, we have%
\[
\sum_{i>1}w^{ii}\mathsf{MTW}_{ii}^{11}w_{11}^{2}\geq\sum_{i>1}w^{ii}%
\delta\left\Vert \partial x_{i}\right\Vert _{g}^{2}\left\Vert w_{11}%
dx_{1}\right\Vert _{g}^{2},
\]
thus (\ref{mtwsum})\ is bounded by
\begin{equation}
C\sum w_{ii}-\delta\sum_{i>1}w^{ii}w_{11}^{2}\leq\sum w_{ii}-\delta c_{n}%
\sum_{i}w^{ii}w_{11}^{2} \label{jjputz}%
\end{equation}
as our chart for the source $M$ is normal at $x_{0}.$ Since $\det w_{ij}$ is
bounded it follows from the arithmetic-geometric mean that
\[
c\left(  \sum w_{ii}\right)  ^{1/(n-1)}\leq\sum w^{ii}\leq C\left(  \sum
w_{ii}\right)  ^{n-1}.
\]
So finally, from Lemma \ref{secondderivativelemma}
\begin{align*}
\frac{\partial}{\partial t}w_{vv}-\left(  Lw\right)  _{vv}  &  \leq C\sum
w_{ii}-\delta c_{n}\sum_{i}w^{ii}w_{11}^{2}.\\
&  +C(1+\sum w_{ii}^{2}+\sum w^{ii}\sum w_{jj})\\
&  \leq-\frac{\delta}{2}c_{n}\sum_{i}w^{ii}w_{11}^{2}+C\sum w_{ii}+C\sum
w_{ii}^{2}\\
&  +C\sum w^{ii}(\sum w_{jj}-\frac{\delta}{2}c_{n}w_{11}^{2})\\
&  \leq-\frac{\delta}{2}c_{n}\left(  \sum w_{ii}\right)  ^{1/(n-1)}%
c_{n}\left(  \sum w_{ii}\right)  ^{2}+C\left(  \sum w_{ii}\right)  ^{2}\\
&  +C\sum w^{ii}\left\{  \sum w_{ii}-\frac{\delta}{2}c_{n}\left(  \sum
w_{ii}\right)  ^{2}\right\}  .
\end{align*}
We see that when $\sum w_{ii}$ is sufficiently large, the right-hand side must
be negative. This completes the proof.
\end{proof}

\begin{corollary}
\label{C:strict c-convex} In the situation of Theorem~\ref{Hessian}, the
spatial second derivatives of the solution $u$ to the parabolic equation
\eqref{eq:main} remain bounded (uniformly in time) and $u$ stays locally
strictly $c$-convex. \ 
\end{corollary}

\begin{proof}
An upper bound on the eigenvalues of $w_{ij}$ is given in
Theorem~\ref{Hessian}. From the identity $w_{ij}=u_{ij}+c_{ij},$ the bound on
$u_{ij}$ follows. Because $\theta$ remains bounded by parabolic maximum
principe for $\theta_{t}=L\theta$, we have a lower bound on determinant $\det
w_{ij}$. An upper bound on eigenvalues, plus lower bound on determinant
implies lower bound on the eigenvalues. \ It follows that
\[
u_{ij}+c_{ij}\geq\varepsilon g_{ij}.
\]
for some $\varepsilon>0$, thus{ local} strict $c$-convexity of $u$ follows.
\end{proof}

\section{Proof of Theorem~\ref{main}: Existence of solution and uniform
bounds}

\label{S:long time} In this section we show that the solution to parabolic
equation \eqref{eq:main} exists for all $t\in\mathbf{R}_{+}$ under the
assumptions of Theorem~\ref{main}. We also show the solution has uniform
$C^{k}$ derivatives in $x$, where each $C^{k}$ norm is uniform in $t$. Through
this section we use Corollary~\ref{C:uniform bounds} in an essential way. In
the following, all the $C^{1,1},C^{2,\alpha},C^{4,\alpha}$ estimates and so
on, are estimates on the derivatives in $x$, and are all uniform in the time
variable $t$. But, by $C^{\alpha}$ we will mean $C^{\alpha}$ both in $x$ and
$t$.

\subsection*{Short-time existence:}

Since $M$ is a closed manifold, a standard theory implies the existence of a
short-time solution to (\ref{eq:main}) for any{ locally} strictly $c$-convex
smooth initial data, and that from Corollary~\ref{C:strict c-convex} the
solution $u$ is{ locally} strictly $c$-convex on the time interval of
existence. (See \cite{Kt} for a proof{ of short time existence} regarding the
same equation with a more involved boundary condition.)\ 

\subsection*{Long-time existence:}

Apply Theorem~\ref{Hessian} (or Corollary~\ref{C:strict c-convex}) to get
$C^{1,1}$ estimates for $u$. This in particular makes the
equation~\eqref{eq:main} as well as the linearized equation $v_{t}=Lv$ (see
\eqref{eq:L}) uniformly parabolic with bounded coefficients. Now applying
Krylov-Safonov theory (c.f. \cite[Lemma 14.6]{L}) to ~\eqref{eq:main} one has
$C^{2,\alpha\text{ }}$estimates. From the short-time existence above and
Arzela-Ascoli, this shows that the solution cannot cease to exist at a finite
time, thus exists for all $t\in\lbrack0,\infty)$.

\subsection*{Uniform $C^{k}$ bounds:}

To see uniform $C^{k}$ bounds, first differentiate \eqref{eq:main} with
respect to any coordinate direction $\frac{\partial}{\partial x^{k}},$ to see
$u_{k}$ satisfies
\begin{equation}
u_{kt}=Lu_{k}. \label{eq:schauder}%
\end{equation}
From $C^{2,\alpha}$ estimates of $u$, this linear equation for $u_{k}$ is
uniformly parabolic with $C^{\alpha}$ controlled coefficients (here
$C^{\alpha}$ control in $t$ follows from the parabolic equations like
\eqref{eq:schauder} with space $C^{2,\alpha}$ estimates), and in particular,
we can apply parabolic Schauder estimates to conclude that $u_{k}$ has
$C^{2,\alpha}$ estimates, thus obtaining $C^{3,\alpha}$ estimates for $u$.
Similarly, differentiating \eqref{eq:schauder} we obtain a parabolic equation
for $u_{kl}$ also with coefficients and inhomogeneous terms all of which are
$C^{\alpha}$, and thus follows $C^{4,\alpha}$ estimates of $u$, so forth.
Thus, we have uniform $C^{k}$, $k=0,1,\cdots$, bounds for $u$ as claimed.

To obtain the exponential convergence, we need a Harnack inequality, which is
shown in the next sections.

\section{A Li-Yau type Harnack inequality}

\label{S:exponential convergence} { In this section, {as a preliminary step to
the proof of the exponential convergence to a solution of the elliptic
equation \eqref{eq:OT}, we derive a Harnack inequality for the quantity
$\theta$.} We first find a simple expression of the linearized operator $L$ of
$\theta$. This expression allows us to derive a Harnack type estimate (see
Theorem~\ref{harnack}), {whose corollary (Corollary~\ref{C:harnack}) is used
to show the exponential convergence in Section~\ref{S:exponential}. } }

Let us find a convenient expression for the linearized operator $L$, using
\cite[Prop 2.1]{Warren}. \ As a linearized operator of coordinate invariant
fully nonlinear equation \eqref{eq:main}, $L$ is expected to be related \ to a
Laplacian operator of certain Riemannian metric. More precisely, for a
manifold $M$, suppose that a map $T$ is $c$-$\exp du$ for some $u.$ (This
means by definition $T$ is given by \eqref{eq:potential}.) Define a function
$\psi$ by
\begin{equation}
\psi=\left(  \frac{\bar{\rho}^{2}(T(x))\det DT}{\left\vert \det c_{i\bar{s}%
}(x,T(x))\right\vert }\right)  ^{1/(n-2)}.
\end{equation}
{We observe} the following\ 

\begin{proposition}
{\label{LB}} Let $n\geq3.$ The linearized operator {$L$ of $\theta$ (see
\eqref{theta}) is expressed as%
\[
Lv=\psi\Delta_{\psi w}v,
\]
where $\Delta_{\psi w}$ is the Laplace-Beltrami operator with respect to the
metric $\psi w_{ij}$ given on $M$. }
\end{proposition}

\begin{proof}
First of all, let
\begin{equation}
g_{ij}(x)=w_{ij}(x)\left(  \frac{\rho(x)\bar{\rho}(T(x))}{\left\vert \det
c_{i\bar s}(x,T(x))\right\vert }\right)  ^{1/(n-2)} . \label{themetric}%
\end{equation}
Also, recall
\[
\theta=\ln\det DT-\ln\rho(x)+\ln\bar{\rho}(T(x)).
\]
Then \cite[Prop 2.1]{Warren} states that
\[
Lv=\left(  \frac{\rho(x)\bar{\rho}(T(x))}{\det c_{is}(x,T(x))}\right)
^{1/(n-2)}\left(  \triangle_{g}v+\frac{1}{2}\langle\nabla v,\nabla
\theta\rangle_{g}\right)  .
\]
A general formula for conformal metrics shows that if $\widetilde{g} = e^{f}
g$, then
\begin{align*}
e^{f} \widetilde{\Delta} \phi= \Delta\phi+ \frac{n-2}{2} \left<  \nabla\phi,
\nabla f \right>  .
\end{align*}
It follows immediately that
\begin{align*}
L v = \left(  \frac{\rho(x) \bar{\rho}(T(x))}{\det c_{i\bar s}(x, T(x))}
\right)  ^{1/(n-2)} e^{\theta/(n-2)} \Delta_{e^{\theta/(n-2)} g} v.
\end{align*}
Using the expression for $\theta$ we observe that
\begin{align*}
&  \left(  \frac{\rho(x) \bar{\rho}(T(x))}{\det c_{i\bar s}(x, T(x))} \right)
^{1/(n-2)} e^{\theta/(n-2)}\\
&  =\ \left(  \frac{\rho(x) \bar{\rho}(T(x))}{\det c_{i\bar s}(x, T(x))}
\right)  ^{1/(n-2)} \left(  \frac{ (\det DT) \bar{\rho}(T(x))}{\rho(x)}
\right)  ^{1/(n-2)}\\
&  =\ \left(  \frac{\bar{\rho}^{2}(T(x)) \det DT}{\det c_{i\bar s}(x, T(x))}
\right)  ^{1/(n-2)}.
\end{align*}
The result follows.
\end{proof}

Noting that by Lemma \ref{thetaevolution}, and Proposition \ref{LB}
\[
{\theta}_{t}=L\theta=\psi\Delta_{\psi w}\theta
\]
we derive a Harnack estimate {for the operator $L$}, when $n\geq3.$ \ {Here,
the expression $\psi\Delta_{\psi w}$ enables us to easily modify the argument
in \cite{LiYau} to obtain}

\begin{theorem}
[Harnack inequailty]\label{harnack} Let $M$ be a compact manifold of dimension
$n$ and let $g(t)$, $0\leq t<\infty$ be a family of Riemannian metrics { on
$M$} such that
\begin{align}
\frac{1}{C_{0}}g(0)\leq &  \ g(t)\leq C_{0}g(0)\label{metrics are equiv}\\
\left\vert \frac{\partial}{\partial t}g\right\vert \leq &  \ C_{0}\nonumber\\
R_{ij}(t)\geq &  \ -Kg_{ij}(0)\nonumber
\end{align}
{ with universal constants $C_{0},K>0$. Let $\lambda(x,t)$ be a positive
function with derivatives uniformly controlled (independent of $(t,x)$) and
bounded above and away from zero.} \ Let \ $U(x,t)$ be a nonnegative solution
to
\begin{align}
\label{eq:lambda Delta}U_{t}=\lambda(x,t)\Delta_{g(t)}U.
\end{align}
Then there exists a constant $C>0$ depending only on $C_{0},K$, $g_{ij}(0)$
and the bounds on the derivatives of $\lambda$ so that for $0<t_{1}%
<t_{2}<\infty$
\[
\sup_{x\in M}U(x,t_{1})\leq\inf_{x\in M}U(x,t_{2})\,C\frac{t_{2}}{t_{1}}%
\exp\left\{  \frac{C_{0}^{2}diam^{2}(M_{t_{1}})}{(t_{2}-t_{1})}+C(t_{2}%
-t_{1})\right\}  .
\]

\end{theorem}

Note that by the a priori estimates in the previous sections, the metric
$g_{ij}=\psi w_{ij}$ and scalar $\lambda=\psi$ all satisfy the assumptions in
Theorem~\ref{harnack}, thus we obtain the corresponding Harnack inequality for
{the operator $L$}. In particular, we have

\begin{corollary}
\label{C:harnack} For $n\geq3,$ let $U:M\times\lbrack0,\infty]\rightarrow
\mathbf{R}$ be a solution to the parabolic equation $U_{t}=LU$, where $L$ is
the linearized operator in \eqref{eq:L}. There exists a constant independent
of $(x,t)\in M\times\lbrack0,\infty]$ such that
\[
\sup_{x\in M}U(x,t+1/2)\leq C\inf_{x\in M}U(x,t).
\]

\end{corollary}

\section{Proof of Theorem~\ref{harnack} (Harnack inequality)}

\label{S:harnack} {This whole section is devoted to the proof of the Harnack
inequality claimed in Theorem~\ref{harnack}. Since the equation {
\eqref{eq:lambda Delta}} is not a pure heat equation, but conformally related
to one for a time dependant metric, we are forced to reprove the Harnack
estimate for this operator. The following argument is a slight modification of
that found in \cite{LiYau}. } \ \ Let%
\[
f=\log U,
\]
then
\[
f_{t}=\lambda(\Delta f+|\nabla f|^{2}).
\]
Now let
\[
F=t\left(  \lambda|\nabla f|^{2}-\alpha f_{t}\right)  .
\]
We directly compute {
\begin{align*}
\Delta F\geq &  \ t\left(
\begin{array}
[c]{c}%
\lambda\triangle|\nabla f|^{2}+2\langle\nabla\lambda,\nabla|\nabla
f|^{2}\rangle+\triangle\lambda|\nabla f|^{2}-\alpha\left(  \Delta f\right)
_{t}\\
-\alpha\left\vert \frac{\partial g}{\partial t}\right\vert \left(  \left\vert
\nabla^{2}f\right\vert +\left\vert \nabla f\right\vert \right)
\end{array}
\right) \\
\geq &  \ t\left(
\begin{array}
[c]{c}%
2\lambda\left\Vert \nabla^{2}f\right\Vert ^{2}+2\lambda\left\langle \nabla
f,\nabla\Delta f\right\rangle -2\lambda K|\nabla f|^{2}-\varepsilon\left\Vert
\nabla^{2}f\right\Vert ^{2}-\frac{C_{1}}{\varepsilon}|\nabla f|^{2}\\
-C_{2}|\nabla f|^{2}-\alpha\left(  \Delta f\right)  _{t}-\varepsilon\left\Vert
\nabla^{2}f\right\Vert ^{2}-\frac{C_{3}}{\varepsilon}%
\end{array}
\right) \\
&  \geq\ t\left(  \gamma\left\Vert \nabla^{2}f\right\Vert ^{2}+2\lambda
\left\langle \nabla f,\nabla\Delta f\right\rangle -\alpha\left(  \Delta
f\right)  _{t}-C_{4}|\nabla f|^{2}-C_{3}\right)  .
\end{align*}
} In the second line we applied the Bochner formula, with $K$ a uniform lower
bound on the time-dependent Ricci curvature, \ $C_{1}$ is an upper bound on
$\left\Vert \nabla\lambda\right\Vert ,$ $\varepsilon$ is somewhat smaller than
a lower bound on $\lambda,$ \ $C_{2}$ bounds $\triangle\lambda$ and the
constant $C_{3\text{ }}$ bounds the time derivative of the metric. \ \ The
constant $\gamma$ is a positive lower bound for $\lambda$. \ 

Now
\begin{equation}
\Delta f=-|\nabla f|^{2}+\frac{1}{\lambda}f_{t}=-\frac{1}{\lambda}\left(
\frac{1}{t}F-(1-\alpha)f_{t}\right)  . \label{LY1}%
\end{equation}
We thus further estimate%
\begin{align*}
\Delta F\geq &  \ \ \ t\left(
\begin{array}
[c]{c}%
\gamma\left\Vert \nabla^{2}f\right\Vert ^{2}-2\lambda\left\langle \nabla
f,\nabla\frac{1}{\lambda}\left(  \frac{1}{t}F-(1-\alpha)f_{t}\right)
\right\rangle +\\
\alpha\left[  \frac{1}{\lambda}\left(  \frac{1}{t}F-(1-\alpha)f_{t}\right)
\right]  _{t}-C_{4}|\nabla f|^{2}-C_{3}%
\end{array}
\right) \\
&  \geq\ \ t\gamma\left\Vert \nabla^{2}f\right\Vert ^{2}-2\left\langle \nabla
f,\nabla F\rangle+2t(1-\alpha)\langle\nabla f,\nabla f_{t}\right\rangle
-C_{5}t\left\vert \nabla f\right\vert \left\vert (\lambda\triangle
f)\right\vert \\
&  +\frac{1}{\lambda}\left(  \alpha F_{t}-\alpha\frac{F}{t}-t\alpha
(1-\alpha)f_{tt}\right)  -C_{6}t\left\vert (\lambda\triangle f)\right\vert
-C_{4}t|\nabla f|^{2}-C_{3}t
\end{align*}%
\begin{align*}
&  \geq\ \ t\gamma_{2}\left\Vert \nabla^{2}f\right\Vert ^{2}-2\langle\nabla
f,\nabla F\rangle+\frac{1}{\lambda}\left\{  2t\lambda(1-\alpha)\langle\nabla
f,\nabla f_{t}\rangle-t\alpha(1-\alpha)f_{tt}\right\} \\
&  +\frac{1}{\lambda}\left(  \alpha F_{t}-\alpha\frac{F}{t}\right)
-C_{7}t|\nabla f|^{2}-C_{3}t\\
&  =\ t\gamma_{2}\left\Vert \nabla^{2}f\right\Vert ^{2}-2\langle\nabla
f,\nabla F\rangle+\frac{1}{\lambda}\left\{  (1-\alpha)F_{t}-(1-\alpha)\frac
{F}{t}-\lambda tC_{3}|\nabla f|^{2}\right\} \\
&  +\frac{1}{\lambda}\left(  \alpha F_{t}-\alpha\frac{F}{t}\right)
-C_{7}t|\nabla f|^{2}-C_{3}t\\
&  \geq\ t\gamma_{2}\left\Vert \nabla^{2}f\right\Vert ^{2}-2\langle\nabla
f,\nabla F\rangle+\frac{1}{\lambda}F_{t}-\frac{1}{\lambda}\frac{F}{t}%
-tC_{9}|\nabla f|^{2}-C_{3}t\\
&  \geq\ t\frac{\gamma_{2}}{\lambda n}\left(  -\lambda|\nabla f|^{2}+f_{t}%
^{{}}\right)  ^{2}-2\langle\nabla f,\nabla F\rangle+\frac{1}{\lambda}%
F_{t}-\frac{1}{\lambda}\frac{F}{t}-tC_{9}|\nabla f|^{2}-C_{3}t.
\end{align*}
\newline Here $C_{5}$ bounds $\left\vert \nabla\frac{1}{\lambda}\right\vert $,
$\ $\ and $C_{6}$ bounds $\left\vert \frac{1}{\lambda}\right\vert _{t},$
\ with $\gamma_{2}$ an even smaller constant, related to $C_{7}.$ \ In the
last line we applied $\left\Vert \nabla^{2}f\right\Vert ^{2}\geq\frac{1}%
{n}(\Delta f)^{2}$ and the relation (\ref{LY1}).

Now at a maximum for $F$ which happens at some positive time on $M\times
\lbrack0,T]$ we can conclude%

\begin{align}
0\geq &  \ t\frac{\gamma_{2}}{\lambda n}\left(  -\lambda|\nabla f|^{2}%
+f\right)  ^{2}-2\langle\nabla f,\nabla F\rangle+\frac{1}{\lambda}F_{t}%
-\frac{1}{\lambda}\frac{F}{t}-tC_{9}|\nabla f|^{2}-C_{3}t\nonumber\\
\geq &  \ \ \ t\gamma_{3}\left(  -\lambda|\nabla f|^{2}+f_{t}\right)
^{2}-\frac{1}{\lambda}\frac{F}{t}-tC_{9}|\nabla f|^{2}-C_{3}t. \label{LY18}%
\end{align}
Now let
\begin{align*}
y  &  =\lambda|\nabla f|^{2},\\
z  &  =f_{t}.
\end{align*}
Continuing to follow \cite[Eq 1.9]{LiYau} we expand
\begin{equation}
(y-z)^{2}=\left(  \frac{1}{\alpha}-\frac{\varepsilon}{2}\right)  (y-\alpha
z)^{2}+(1-\varepsilon-\delta-\frac{1}{\alpha}+\frac{\varepsilon}{2}%
)y^{2}+(1-\alpha+\frac{\varepsilon}{2}\alpha^{2})z^{2}+\varepsilon y(y-\alpha
z)+\delta y^{2} \label{LY19}%
\end{equation}
Choose
\begin{align*}
\varepsilon &  =2-2\frac{1}{\alpha}-2\frac{1}{(1-\alpha)^{2}}>0\\
\delta &  =\frac{1}{(1-\alpha)^{2}}%
\end{align*}
and $\alpha$ so that
\begin{align*}
\frac{1}{\alpha}  &  >\varepsilon>0\\
\delta &  >\frac{\varepsilon}{2}>0.
\end{align*}
Specifically, one may choose $\alpha=\frac{5}{2}$ to satisfy these conditions.
Note that the second and third terms in (\ref{LY19}) vanish, and multiplying
(\ref{LY18}) by $t$, absorbing bounds on $\lambda$ and combining with
(\ref{LY19})%
\[
0\geq\ \ \ t^{2}\gamma_{4}\left\{  \left(  \frac{1}{\alpha}-\frac{\varepsilon
}{2}\right)  (y-\alpha z)^{2}+\varepsilon y(y-\alpha z)+\delta y^{2}\right\}
-F-t^{2}C_{10}y-C_{11}t^{2}%
\]%
\[
\geq\ \ \ \gamma_{4}\left\{  \left(  \frac{1}{\alpha}-\frac{\varepsilon}%
{2}\right)  F^{2}-\frac{\varepsilon}{2}F^{2}-\frac{\varepsilon}{2}t^{2}%
y^{2}+\delta t^{2}y^{2}\right\}  -F-t^{2}C_{10}y-C_{11}t^{2}%
\]

\[
\geq\ \ \ \ \gamma_{4}\left(  \frac{1}{\alpha}-\varepsilon\right)
F^{2}+\ \gamma_{4}\left(  \delta-\frac{\varepsilon}{2}\right)  t^{2}%
y^{2}-F-t^{2}\left[  \frac{\left(  C_{2}+K\right)  }{\ \gamma_{4}\left(
\delta-\frac{\varepsilon}{2}\right)  }+\ \gamma_{4}\left(  \delta
-\frac{\varepsilon}{2}\right)  y^{2}\right]  -C_{11}t^{2}.
\]
From which we conclude that
\[
F\leq\frac{1}{2\gamma_{4}\left(  \frac{1}{\alpha}-\varepsilon\right)  }%
+\frac{\sqrt{1+\left(  4\frac{C_{10}}{\gamma_{4}\left(  \delta-\frac
{\varepsilon}{2}\right)  }\ +C_{11}\right)  \gamma_{4}\left(  \frac{1}{\alpha
}-\varepsilon\right)  }t^{2}}{\gamma_{4}\left(  \frac{1}{\alpha}%
-\varepsilon\right)  }%
\]
i.e
\begin{equation}
\lambda\left\vert \nabla f\right\vert ^{2}-\alpha f_{t}\leq\frac{C}{t}+C.
\label{harnackloc10}%
\end{equation}
Now we have arrived at this conclusion assuming that the maximum happens away
from $t=0.$ \ Letting $f=\log(U+\iota)$ this assumption is available. \ We
then take $\iota\rightarrow0$. \ 

Now consider now a path $\gamma:[0,1]\rightarrow M\times\lbrack t_{1},t_{2}]$
such that $\gamma(0)=(y,t_{2})$ and $\gamma(1)=(x,t_{1})$, which for
convenience, projects to a geodesic in $M_{t_{1}}$ and has constant speed in
$t.$ \ Using the\ assumption (\ref{metrics are equiv}) and (\ref{harnackloc10}%
)
\begin{align*}
f(x,t_{1})-f(y,t_{2})\leq &  \ \int_{0}^{1}\left\{  |\nabla f|C_{0}%
d(x,y)-(t_{2}-t_{1})(f)_{t}\right\}  ds\\
\leq &  \ \int_{0}^{1}\left\{  |\nabla f|C_{0}d(x,y)+(t_{2}-t_{1})\left(
C+\frac{C}{t}-\frac{\lambda|\nabla f|^{2}}{\alpha}\right)  \right\}  ds\\
\leq &  \ \int_{0}^{1}\left\{  \frac{\alpha}{4(t_{2}-t_{1})\lambda}C_{0}%
^{2}d^{2}(x,y)+(t_{2}-t_{1})\left(  C+\frac{C}{t}\right)  \right\}  ds\\
\leq &  \ \frac{\alpha}{4(t_{2}-t_{1})\min\lambda}C_{0}^{2}d^{2}%
(x,y)+C(t_{2}-t_{1})+C\log\frac{t_{2}}{t_{1}}%
\end{align*}
where in the last line we integrate using that $t=(1-s)t_{2}+st_{1}$. {This
completes the proof of Harnack inequality (Theorem~\ref{harnack}).}

\section{ Proof of Theorem~\ref{main}: Exponential convergence to the solution
to optimal transport problem}

\label{S:exponential} {This section completes the proof of Theorem~\ref{main}.
We assume the conditions in Theorem~\ref{main}. In the first subsection we
show the exponential convergence to a stationary solution, and then in the
last subsection we show that the stationary solution is indeed the solution to
the optimal transport problem.}

\subsection{Exponential convergence}

\subsubsection{Case $n\geq3$}

With this parabolic Harnack inequality at hand, the claimed exponential
convergence to the optimal transportation map follows from a rather standard
argument. To see this, consider for $k\in\mathbf{N}$,
\begin{align*}
U_{k}=  &  \ \sup_{x\in M}\theta(x,k-1)-\theta(x,(k-1)+t),\\
L_{k}=  &  \theta(x,k-1+t)-\inf_{x\in M}\theta(x,k-1)
\end{align*}
that are also solutions to the equation $U_{t}=LU$. By the strong maximum
principle, both $U$ and $L$ are positive functions for positive $t,$ for all
$k$. Also, let $H(t)=\sup_{x\in M}\theta(x,t)-\inf_{x\in M}\theta(x,t)$. The
Harnack inequality in {Theorem~\ref{harnack} yields }
\begin{align*}
\sup\theta(x,k-1)-\inf\theta(x,k-\frac{1}{2})  &  \leq C\left(  \sup
\theta(x,k-1)-\sup\theta(x,k)\right)  ,\\
\sup\theta(x,k-\frac{1}{2})-\inf\theta(x,k-1)  &  \leq C\left(  \inf
\theta(x,k)-\inf\theta(x,k-1)\right)
\end{align*}
for some fixed constant $C>1$. It follows by adding the equations together
that
\[
H(k-1)+H(k-\frac{1}{2})\leq C\left(  H(k-1)-H(k)\right)
\]
which implies
\[
H(k)\leq\epsilon H(k-1)
\]
where $\epsilon=\frac{C-1}{C}<1$. By induction we observe
\[
H(k)\leq\epsilon^{k}H(0).
\]
It follows that $H(t)\leq Ce^{-\beta t}$ where $\epsilon=e^{-\beta}$.
Therefore, $\theta$ converges to the limit $\theta_{\infty}\equiv const.$
exponentially fast as $t\rightarrow\infty$. Now the quantity $\theta$ can be
larger than $0$ somewhere, only if it is smaller than $0$ somewhere, as can be
seen by integrating the change of measures and using $\int_{M}\rho
(x)dx=\int_{\bar{M}}\bar{\rho}(\bar{x})d\bar{x}$. It follows that
$\theta_{\infty}\equiv0$ and thus from $u_{t}=\theta$, we see $u_{t}%
\rightarrow0$ exponentially fast as $t\rightarrow\infty$. This implies $u$
converges exponentially fast to a stationary solution $u_{\infty}$, which is
smooth from the uniform $C^{m}$, $m\in\mathbf{N}$, estimates. Because
$\theta_{\infty}\equiv0$, $u_{\infty}$ solves the elliptic equation
\eqref{eq:OT}. \ Considering the discussion in Section \ref{SS:loc2glob}, this
solution is a solution to the optimal transportation problem. This finishes
the proof of the claimed exponential convergence and of Theorem~\ref{main} for
$n\geq3.$

\subsubsection{Case $n=2$}

On $M^{2}$ $\times\bar{M}^{2}$ a solution $u$ will satisfy all of the
estimates, hence exists and enjoys subsequential convergence at infinity.
\ The only missing piece is the Harnack inequality. \ However, we can fake a
third dimension and get a solution on $\ M^{2}\times S^{1}$ $\rightarrow
\bar{M}^{2}\times S^{1}$ by letting $\tilde{u}(x,z)=u(x)$ and taking products
of the measures with uniform measures on $S^{1}.$ \ Hence $\tilde{u}(x,z)$
will also be a solution, and will converge in the same way to the three
dimensional product solution. \ 

\subsection{The limit stationary solution is the solution to the optimal
transport problem: strict global $c$-convexity:}

\label{SS:loc2glob} To conclude that the limit stationary solution, say
$u_{\infty}$, is a solution to the optimal transport problem (and not a
spurious solution to the elliptic equation), it remains to show that
$u_{\infty}$ is globally strict $c$-convex, which from the discussion in the
middle of Section~\ref{SS:stay away} follows if the corresponding map, say
$T_{\infty}$, is a global diffeomorphism. (This map $T_{\infty}$ is already a
local diffeomorphism by local strict $c$-convexity of $u_{\infty}$.) To see
this, we use that $u_{\infty}$ satisfies the Monge-Amp\`{e}re type equation
\eqref{eq:OT}. If $T_{\infty}$ is not one-to-one, as a local diffeomorphism
between closed manifolds, it is a covering map, having the topological degree
greater than $1$. Thus, from \eqref{eq:OT} the push-forward $T_{\infty\#}\rho$
is a multiple of the target measure $\bar{\rho}$. But, this contradicts the
fact that $\int_{M}\rho=\int_{\bar{M}}\bar{\rho}$ since $\int_{\bar{M}%
}T_{\infty\#}\rho=\int_{M}\rho$. This finishes the proof of the fact that the
limiting stationary solution of \eqref{eq:main} is the solution to the optimal
transport problem, and thus together with all the previous sections
(especially Section~\ref{S:long time}) it completes the proof of
Theorem~\ref{main}.

As a final remark, we state a corollary to Theorem~\ref{main}, in particular
this last paragraph:

\begin{corollary}
\label{C:local to global c-convex} Assume that the same conditions as in
Theorem~\ref{main} hold. Then, any locally strictly $c$-convex $C^{2}$
function $u_{0}$ is in fact globally $c$-convex.
\end{corollary}

\begin{proof}
Let $T_{0}$ denote the corresponding map to $u_{0}$ by the formula
\eqref{eq:potential}. Local strict $c$-convexity implies that $T_{0}$ is local
diffeomorphism and for global $c$-convexity of $u_{0}$, it is enough to show
that $T_{0}$ has topological degree $1$. From Theorem~\ref{main}, the map $T$
depends continuously on the time variable $t$. In particular, the degree stays
constant. From the result of Theorem~\ref{main} the map $T_{\infty}$ in the
lim $t\to\infty$ is a diffeomorphism and so its degree is $1$. This shows that
the degree of $T_{0}$ is also $1$, and completes the proof that $u_{0}$ is
globally $c$-convex.
\end{proof}

\end{document}